\documentclass[11pt]{amsart}
\usepackage[margin=1in]{geometry}
\usepackage{latexsym}
\usepackage{amsfonts}
\usepackage{amsmath}
\usepackage{amssymb}
\usepackage{amsthm}
\usepackage{enumerate}
\usepackage[hang,flushmargin]{footmisc}
\usepackage{caption}
\usepackage{mathrsfs}
\usepackage{amsaddr}
\usepackage{subfig}

\newcommand{\Av}{\operatorname{Av}}

\usepackage{graphicx}
\usepackage{epstopdf}
\usepackage{epsfig}
\usepackage{caption}
 
\usepackage{bm} 
\usepackage{cite}

\makeatletter
\newtheorem*{rep@theorem}{\rep@title}
\newcommand{\newreptheorem}[2]{%
\newenvironment{rep#1}[1]{%
 \def\rep@title{#2 \ref{##1}}%
 \begin{rep@theorem}}%
 {\end{rep@theorem}}}
\makeatother

\newtheorem{theorem}{Theorem}[section]
\newtheorem{lemma}[theorem]{Lemma}
\newtheorem{proposition}[theorem]{Proposition}
\newtheorem{corollary}[theorem]{Corollary}

\newtheorem{conjecture}[theorem]{Conjecture}

\theoremstyle{definition}
\newtheorem{definition}[theorem]{Definition}
\newtheorem{remark}[theorem]{Remark}

\DeclareMathOperator{\VHC}{\mathsf{VHC}}

\DeclareMathOperator{\act}{act}
\DeclareMathOperator{\tl}{tl}
\DeclareMathOperator{\Int}{Int}
\DeclareMathOperator{\comp}{comp}
\DeclareMathOperator{\SW}{SW}
\DeclareMathOperator{\cl}{cl}

\DeclareMathOperator{\DL}{\Lambda\hspace{-.095cm}\Lambda}
\DeclareMathOperator{\lon}{long}

\begin{document}
\title{Motzkin Intervals and Valid Hook Configurations}
\author{Colin Defant}
\address{Princeton University \\ Fine Hall, 304 Washington Rd. \\ Princeton, NJ 08544}
\email{cdefant@princeton.edu}

\begin{abstract}
We define a new natural partial order on Motzkin paths that serves as an intermediate step between two previously-studied partial orders. We provide a bijection between valid hook configurations of $312$-avoiding permutations and intervals in these new posets. We also show that valid hook configurations of permutations avoiding $132$ (or equivalently, $231$) are counted by the same numbers that count intervals in the Motzkin-Tamari posets that Fang recently introduced, and we give an asymptotic formula for these numbers. We then proceed to enumerate valid hook configurations of permutations avoiding other collections of patterns. We also provide enumerative conjectures, one of which links valid hook configurations of $312$-avoiding permutations, intervals in the new posets we have defined, and certain closed lattice walks with small steps that are confined to a quarter plane.   
\end{abstract}

\maketitle

\bigskip

\section{Introduction}\label{Sec:Intro}

\subsection{Partial Orders on Motzkin Paths}

A \emph{Motzkin path} is a lattice path consisting of $(1,1)$ steps (called \emph{up steps}), $(1,-1)$ steps (called \emph{down steps}), and $(1,0)$ steps (called \emph{east steps}) that starts at the origin, ends on the horizontal axis, and never passes below the horizontal axis. Let $U,D,E$ represent up, down, and east steps, respectively. We can think of a Motzkin path $\Lambda$ of length $n$ as a word $\Lambda_1\cdots\Lambda_n$ of length $n$ over the alphabet $\{U,D,E\}$ that has as many $U$'s as it has $D$'s and also has the property that each of its prefixes has at least as many $U$'s as $D$'s. The number of Motzkin paths of length $n$ is the $n^\text{th}$ Motzkin number $M_n$ (OEIS sequence A001006). Let ${\bf M}_n$ be the set of Motzkin paths of length $n$. A \emph{Dyck path} is a Motzkin path that has no east steps. Let ${\bf D}_k$ be the set of Dyck paths of length $2k$. 

There is a natural partial order $\leq_S$ on ${\bf M}_n$ that we obtain by declaring that $\Lambda\leq_S\Lambda'$ if $\Lambda$ lies below or is equal to $\Lambda'$. Alternatively, we have $\Lambda_1\cdots\Lambda_n\leq_S\Lambda_1'\cdots\Lambda_n'$ if and only if the number of $U$'s in $\Lambda_1\cdots\Lambda_i$ is at most the number of $U$'s in $\Lambda_1'\cdots\Lambda_i'$ for every $i\in[n]$. When $n=2k$, this order relation induces a poset on the subset ${\bf D}_k\subseteq {\bf M}_n$. Among other results, Ferrari and Pinzani \cite{Ferrari} proved that the posets $\mathcal L_k^S:=({\bf D}_k,\leq_S)$ and $\mathcal M_n^S:=({\bf M}_n,\leq_S)$ are lattices. Bernardi and Bonichon \cite{Bernardi} called $\mathcal L_k^S$ the $k^\text{th}$ \emph{Stanley lattice}. By analogy, we call $\mathcal M_n^S$ the $n^\text{th}$ \emph{Motzkin--Stanley lattice}. 
 
The $k^\text{th}$ Tamari lattice, which we denote by $\mathcal L_k^T$, is an extremely important sublattice of the $k^\text{th}$ Stanley lattice. Tamari lattices have seen a huge amount of attention from researchers in combinatorics, group theory, theoretical computer science, algebraic geometry, and algebraic topology \cite{Bernardi, Chapoton, Combe, Csar, Early, Fang2, Geyer, Huang, Knuth2, Loday, Pallo2, Tamari}. Recently, Fang introduced new posets defined on the sets ${\bf M}_n$ that are analogous to the Tamari lattices. He investigated the structural and enumerative aspects of the components and intervals of these posets. We denote these posets, which we define formally in Section \ref{Sec:Motzkin}, by $\mathcal M_n^T$. 

In Section \ref{Sec:Motzkin}, we define new posets $\mathcal M_n^C$ that are natural intermediate steps between the Motzkin-Stanley lattices $\mathcal M_n^S$ and the Motzkin-Tamari posets $\mathcal M_n^T$. More precisely, $\mathcal M_n^T$ is a subposet of $\mathcal M_n^C$, which in turn is a subposet of $\mathcal M_n^S$. The intervals in the lattices $\mathcal L_k^S$ and $\mathcal L_k^T$ have been the subject of recent investigations \cite{Bernardi, Combe, De, DefantCatalan}. One of our main goals in this paper is to link the intervals in the posets $\mathcal M_n^C$ and $\mathcal M_n^T$ with recently-introduced combinatorial objects called ``valid hook configurations." We define these objects formally in Section \ref{Sec:VHCs}, but roughly speaking, they are configurations of L-shaped ``hooks" drawn on permutations that satisfy particular constraints. 

\subsection{Valid Hook Configurations}

The original motivation for studying valid hook configurations comes from a formula in the study of West's stack-sorting map. We refer to \cite{Bona, BonaSurvey, DefantCatalan, DefantCounting, DefantTroupes, DefantEngenMiller} for the definition of this map and additional information about it. West \cite{West} defined the \emph{fertility} of a permutation $\pi$ to be the number of preimages of $\pi$ under the stack-sorting map, and he computed the fertilities of some very specific permutations. Bousquet-M\'elou \cite{Bousquet} found an algorithm to determine whether or not the fertility of a permutation is $0$, and she asked for a general method for computing the fertility of an arbitrary permutation. This was accomplished in \cite{DefantPostorder,DefantPreimages}, where the current author found a formula for the fertility of a permutation $\pi$ as a sum over the valid hook configurations of $\pi$. See also \cite{Fertilitopes, FertilityMono, Polyurethane, StackCoxeter, DefantZheng, FertilityNumbers, StackPreimages}

More recently, the author has found a new formula involving a sum over all valid hook configurations of permutations of length $n-1$ that converts from free cumulants to classical cumulants in noncommutative probability theory \cite{DefantTroupes}. There have been several recent papers devoted to finding combinatorial formulas that convert from one type of cumulant sequence to another \cite{Arizmendi, Belinschi, Celestino, Ebrahimi, Lehner, Josuat}, and it is very surprising that valid hook configurations show up naturally in one such formula. This connection between valid hook configurations and free probability theory is not only unexpected, but is also very useful. Indeed, by combining the fertility formula with the formula that converts between cumulants, the author has been able to apply tools from free probability theory to prove several new surprising propoerties of the stack-sorting map \cite{DefantTroupes}. It seems that the connection with free probability, which passes directly through valid hook configurations, is responsible for much of the unexpected hidden structure lying beneath the stack-sorting map. 

There are two other directions that the paper \cite{DefantTroupes} takes with valid hook configurations. The first vastly generalizes the fertility formula by defining some natural sets of trees called \emph{troupes} and showing that valid hook configurations allow one to enumerate the trees in a troupe that have a given permutation as their postorder readings. The second direction provides yet another formula that converts from free to classical cumulants; this formula, however, is given by a sum over $231$-avoiding valid hook configurations. The fact that $231$-avoiding valid hook configurations appear naturally in this cumulant conversion formula motivates us to understand the combinatorial properties of these objects; we enumerate them in Section~\ref{Sec:132}. More generally, one can consider valid hook configurations avoiding other patterns; this is the goal of the current article.   

The authors of \cite{DefantEngenMiller} found that valid hook configurations (not avoiding any patterns) are enumerated by the absolute values of the classical cumulants of the free Poisson law with rate $-1$. In \cite{DefantTroupes}, this was used to show that
\[\sum_{n\geq 1}|\VHC(S_{n-1})|\frac{x^n}{n!}=-\log\left(1-x\,{}_1\hspace{-.03cm}F_2\left(\frac{1}{2};\frac{3}{2},2;-x^2\right)\right),\] where $\VHC(S_{n-1})$ is the set of valid hook configurations of permutations in $S_{n-1}$ and ${}_1\hspace{-.03cm}F_2$ denotes a generalized hypergeometric function. The article \cite{DefantEngenMiller} shows that uniquely sorted permutations, which are permutations with a unique preimage under the stack-sorting map, are counted by Lassalle's sequence, a fascinating new sequence that was introduced in \cite{Lassalle}  (sequence A180874 in \cite{OEIS}). The terms in Lassalle's sequence are the absolute values of the classical cumulants of the standard semicircular distribution. It turns out that counting uniquely sorted permutations in $S_n$ is equivalent to counting valid hook configurations of permutations in $S_n$ with $\frac{n-1}{2}$ hooks. Thus, results that enumerate certain uniquely sorted permutations can be interpreted as results that enumerate certain valid hook configurations. The paper \cite{DefantCatalan} produced bijections between uniquely sorted permutations that avoid certain patterns and intervals in posets defined on Dyck paths, and Mularczyk \cite{Hanna} enumerated even more sets of pattern-avoiding uniquely sorted permutations. Thus, the papers \cite{DefantCatalan, DefantEngenMiller, Hanna} have counted valid hook configurations of unrestricted permutations in $S_n$, valid hook configurations of permutations in $S_n$ with $\frac{n-1}{2}$ hooks, and valid hook configurations of pattern-avoiding permutations in $S_n$ with $\frac{n-1}{2}$ hooks. What is missing, which is the focus of the current article, is the investigation of valid hook configurations of pattern-avoiding permutations in $S_n$ with no restriction on the number of hooks. 

\subsection{Notation and Terminology}
A \emph{permutation} is an ordering of a set of positive integers, which we write as a word. Let $S_n$ be the set of permutations of the set $[n]=\{1,\ldots,n\}$. If $\pi$ is a permutation of length $n$, then the \emph{normalization} of $\pi$ is the permutation in $S_n$ obtained by replacing the $i^\text{th}$-smallest entry in $\pi$ with $i$ for all $i\in[n]$. Given $\tau\in S_m$, we say a permutation $\sigma=\sigma_1\cdots\sigma_n$ \emph{contains the pattern} $\tau$ if there exist indices $i_1<\cdots<i_m$ in $[n]$ such that the normalization of $\sigma_{i_1}\cdots\sigma_{i_m}$ is $\tau$. We say $\sigma$ \emph{avoids} $\tau$ if it does not contain $\tau$. Let $\Av_n(\tau^{(1)},\ldots,\tau^{(r)})$ denote the set of permutations in $S_n$ that avoid the patterns $\tau^{(1)},\ldots,\tau^{(r)}$. Let $\Av(\tau^{(1)},\ldots,\tau^{(r)})=\bigcup_{n\geq 0}\Av_n(\tau^{(1)},\ldots,\tau^{(r)})$. 

We let $\VHC(\pi)$ denote the set of valid hook configurations of a permutation $\pi$ (defined in Section~\ref{Sec:VHCs}). Given a set $A$ of permutations, let $\VHC(A)=\bigcup_{\pi\in A}\VHC(\pi)$. Define the \emph{tail length} of a permutation $\pi=\pi_1\cdots\pi_n\in S_n$, denoted $\tl(\pi)$, to be the smallest nonnegative integer $\ell$ such that $\pi_{n-\ell}\neq n-\ell$. We make the convention that $\tl(123\cdots n)=n$. If $\tl(\pi)=\ell$, then the \emph{tail} of $\pi$ is the list of points $(n-\ell+1,n-\ell+1),\ldots,(n,n)$. 

An \emph{interval} of a poset $P$ is an ordered pair $(x,y)$ of elements of $P$ such that $x\leq y$. Let $\Int(P)$ denote the set of intervals in the poset $P$. We let $\mathcal M_n^S$, $\mathcal M_n^C$, and $\mathcal M_n^T$ be the posets defined on ${\bf M}_n$ in Section \ref{Sec:Motzkin}. We also let $\mathcal M_n^A$ denote the antichain on ${\bf M}_n$. Note that $|\Int(\mathcal M_n^A)|=|{\bf M}_n|=M_n$. 

\subsection{Summary of Main Results}

In Section \ref{Sec:312}, we produce, for each positive integer $n$, a bijection\footnote{The symbol $\DL$ is pronounced ``double lambda.''} \[\widehat\DL_n:\VHC(\Av_n(312))\to\Int(\mathcal M_{n-1}^C).\] This is an extension of a bijection between $312$-avoiding uniquely sorted permutations (equivalently, valid hook configurations of $312$-avoiding permutations in $S_n$ with $\frac{n-1}{2}$ descents) and intervals in Stanley lattices that was established in \cite{DefantCatalan}. We also give a recurrence that specifies the numbers $|\VHC(\Av_n(312))|$. We make a conjecture that links these numbers with certain lattice walks in the first quadrant that were studied in \cite{Bostan} and \cite{Bousquet5}. In Section \ref{Sec:132}, we prove that \[\sum_{n\geq 1}|\VHC(\Av_n(132))|x^n=\sum_{n\geq 1}|\VHC(\Av_n(231))|x^n=\sum_{n\geq 1}|\Int(\mathcal M_{n-1}^T)|x^n.\] We will see that this generating function is algebraic of degree $5$, and we will derive an asymptotic formula for its coefficients. In Section \ref{Sec:Pairs}, we prove that \[|\VHC(\Av_n(132,231))|=|\VHC(\Av_n(132,312))|=|\VHC(\Av_n(231,312))|=M_{n-1}\] for every positive integer $n$. We can interpret this last result as the statement that certain valid hook configurations are in bijection with intervals of the antichain $\mathcal M_{n-1}^A$. In Sections \ref{Sec:132,321}--\ref{Sec:1243}, we prove the following enumerative results: 
\[\sum_{n\geq 0}|\VHC(\Av_n(132,321))|x^n=\frac{1-3x+3x^2}{(1-x)^4},\] 
\[\sum_{n\geq 0}|\VHC(\Av_n(231,321))|x^n=\frac{1-2x+2x^2-\sqrt{1-4x+4x^2-4x^3+4x^4}}{2x^2},\] \[|\VHC(\Av_n(312,321))|=\sum_{k=0}^{\left\lfloor\frac{n-1}{2}\right\rfloor}\frac{1}{2k+1}{n-k-1\choose k}{n\choose 2k}\quad\text{for every } n\geq 1,\]
\[\sum_{n\geq 0}|\VHC(\Av_n(231,1243))|x^n=1+\frac{2x^2}{3x-1+\sqrt{1-2x-3x^2}}.\] 

The sequences given by these formulas and generating functions are (essentially) the sequences A050407, A025273, A101785, A005773 in the OEIS \cite{OEIS}. According to the OEIS entries, the sequence A101785 counts, among other things, noncrossing partitions whose blocks are all of odd size. The sequence A005773 has numerous combinatorial interpretations. For example, its $n^\text{th}$ term counts directed animals of size $n$, $n$-digit base-$3$ numbers with digit sum $n$, involutions in $S_{2n-2}$ that are invariant under the reverse-complement map and have no decreasing subsequences of length $4$, and minimax elements in the affine Weyl group of the Lie algebra $\mathfrak{so}(2n+1)$.   

\section{Motzkin Posets}\label{Sec:Motzkin}

Every Motzkin path $\Lambda=\Lambda_1\cdots\Lambda_n\in{\bf M}_n$ can be written uniquely in the form \linebreak $X_1D^{\gamma_1}X_2D^{\gamma_2}\cdots X_mD^{\gamma_m}$ for some $X_1,\ldots,X_m\in\{U,E\}$. Note that $n-m$ is the number of $D$'s and also the number of $U$'s appearing in the word. For $j\in\{1,\ldots,m\}$, we define $\lon_j(\Lambda)$ as follows. If $X_j=E$, then $\lon_j(\Lambda)=-1$. If $X_j=U$ and $X_j$ is the $r^\text{th}$ letter in the word $\Lambda$ (so $\sum_{\ell=1}^{j-1}(\gamma_\ell+1)=r-1$), let $\lon_j(\Lambda)$ be the smallest nonnegative integer $t$ such that $\Lambda_{r+1}\cdots\Lambda_{r+t+1}$ contains more $D$'s than $U$'s. We call $(\lon_1(\Lambda),\ldots,\lon_m(\Lambda))$ the \emph{longevity sequence} of $\Lambda$. The \emph{class} of $\Lambda$, denoted $\cl(\Lambda)$, is the set of indices $j\in\{1,\ldots,m\}$ such that $X_j=E$. For example, if $\Lambda=UEDUUEDUDED$ is the Motzkin path in Figure \ref{Fig3}, then the longevity sequence of $\Lambda$ is $(1,-1,6,1,-1,0,-1)$, and the class of $\Lambda$ is $\cl(\Lambda)=\{2,5,7\}$. 

\begin{figure}[h]
\begin{center}
\includegraphics[width=.3\linewidth]{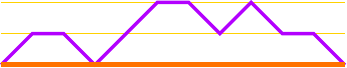}
\caption{The Motzkin path $UEDUUEDUDED$.}
\label{Fig3}
\end{center}  
\end{figure}

Although we will not actually use the definition of the Motzkin-Tamari poset $\mathcal M_n^T$, we state it for the sake of completeness. This definition is not identical to the one given in \cite{Fang2}, but the results in that paper can be used to prove the equivalence of the different definitions. 

\begin{definition}\label{Def1}
Given Motzkin paths $\Lambda,\Lambda'\in{\bf M}_n$, we write $\Lambda\leq_T\Lambda'$ if $\cl(\Lambda)=\cl(\Lambda')$ and $\lon_j(\Lambda)\leq\lon_j(\Lambda')$ for every positive integer $j$ for which $\lon_j(\Lambda)$ and $\lon_j(\Lambda')$ are defined. Let $\mathcal M_n^T$ be the poset $({\bf M}_n,\leq_T)$. 
\end{definition}

The Hasse diagram of $\mathcal M_n^T$ has multiple connected components; two Motzkin paths are in the same component if and only if they have the same class. Fang \cite{Fang2} proved that each of these components is isomorphic to an interval in a classical Tamari lattice. It is straightforward to show that $\Lambda\leq_S\Lambda'$ whenever $\Lambda\leq_T\Lambda'$. In other words, $\mathcal M_n^T$ is a subposet of $\mathcal M_n^S$. We now define a new partial order on ${\bf M}_n$ that serves as a natural intermediate step between $\leq _T$ and $\leq_S$. 

\begin{definition}\label{Def2}
Given Motzkin paths $\Lambda,\Lambda'\in{\bf M}_n$, we write $\Lambda\leq_C\Lambda'$ if $\cl(\Lambda)=\cl(\Lambda')$ and $\Lambda\leq_S\Lambda'$. Let $\mathcal M_n^C$ be the poset $({\bf M}_n,\leq_C)$.
\end{definition}

\section{Valid Hook Configurations}\label{Sec:VHCs}
The \emph{plot} of a permutation $\pi=\pi_1\cdots\pi_n$ is obtained by plotting the points $(i,\pi_i)$ for all $i\in[n]$. A \emph{descent} of $\pi$ is an index $i\in[n-1]$ such that $\pi_i>\pi_{i+1}$. If $i$ is a descent of $\pi$, then we call the point $(i,\pi_i)$ a \emph{descent top of the plot of} $\pi$. 

\begin{figure}[h]
  \centering
  \subfloat[]{\includegraphics[width=0.18\textwidth]{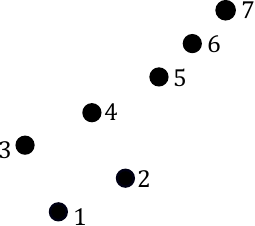}}
  \hspace{1.5cm}
  \subfloat[]{\includegraphics[width=0.18\textwidth]{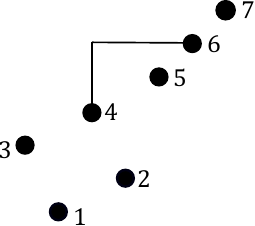}}
  \caption{The left image is the plot of $3142567$. The right image shows this plot along with a single hook.}\label{Fig2}
\end{figure}

A \emph{hook} of $\pi$ is drawn by starting at a point $(i,\pi_i)$ in the plot of $\pi$, drawing a line segment vertically upward, and then drawing a line segment horizontally to the right until reaching another point $(j,\pi_j)$. This only makes sense if $i<j$ and $\pi_i<\pi_j$. The point $(i,\pi_i)$ is called the \emph{southwest endpoint} of the hook, while $(j,\pi_j)$ is called the \emph{northeast endpoint}. The right image in Figure \ref{Fig2} shows the plot of $3142567$ along with a single hook whose southwest endpoint is $(3,4)$ and whose northeast endpoint is $(6,6)$. 

\begin{definition}\label{Def9}
Let $\pi$ be a permutation of length $n$ with $k$ descents, say $d_1<\cdots<d_k$. A \emph{valid hook configuration} of $\pi$ is a tuple $(H_1,\ldots,H_k)$ of hooks of $\pi$ that satisfies the following constraints: 

\begin{enumerate}[1.]
\item For every $i\in[n]$, the southwest endpoint of the hook $H_i$ is the descent top $(d_i,\pi_{d_i})$. 

\item A point in the plot of $\pi$ cannot lie directly above a hook. 

\item Hooks cannot intersect or overlap each other except in the case that the northeast endpoint of one hook is the southwest endpoint of the other. 
\end{enumerate}  
Let $\VHC(\pi)$ denote the set of valid hook configurations of $\pi$. We make the convention that a valid hook configuration includes its underlying permutation as part of its definition. In other words, $\VHC(\pi)$ and $\VHC(\pi')$ are disjoint whenever $\pi$ and $\pi'$ are distinct. Furthermore, we agree that every increasing permutation (including the empty permutation) has exactly one valid hook configuration (which has no hooks). 
\end{definition}

Figure \ref{Fig9} shows four arrangements of hooks that cannot appear in a valid hook configuration. Figure \ref{Fig4} shows all of the valid hook configurations of $3142567$. 

\begin{figure}[h]
\begin{center}
\includegraphics[width=.66\linewidth]{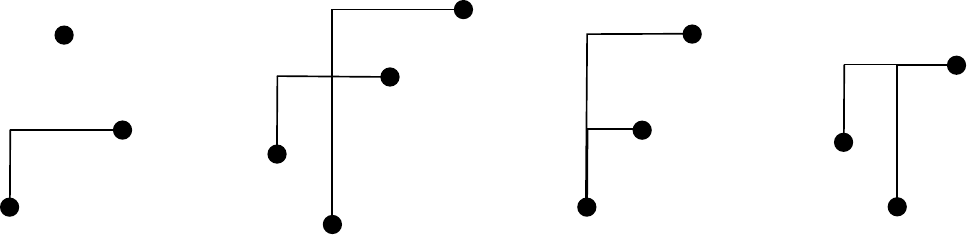}
\caption{Four placements of hooks that are forbidden in a valid hook configuration.}
\label{Fig9}
\end{center}  
\end{figure}

\begin{figure}[h]
\begin{center}
\includegraphics[width=.61\linewidth]{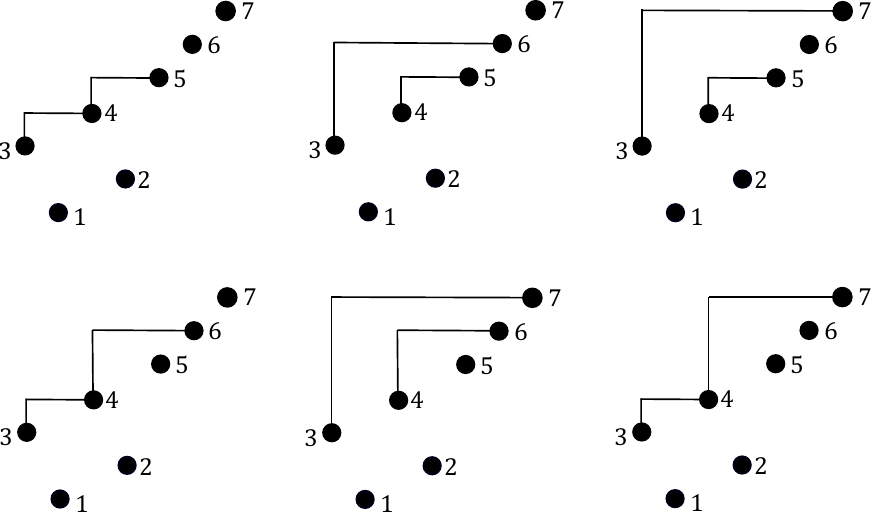}
\caption{The permutation $3142567$ has $6$ valid hook configurations.}
\label{Fig4}
\end{center}  
\end{figure}

Suppose $H$ is a hook of a permutation $\pi$ with southwest endpoint $(i,\pi_i)$ and northeast endpoint $(j,\pi_j)$. Let $\VHC^H(\pi)$ be the set of all valid hook configurations of $\pi$ that include the hook $H$. Assume that $j$ is larger than every descent of $\pi$. The hook $H$ separates $\pi$ into two parts. One part, which we call the \emph{$H$-unsheltered subpermutation of $\pi$} and denote by $\pi_U^H$, is $\pi_1\cdots\pi_i\pi_{j+1}\cdots\pi_n$. The other part, which we call the \emph{$H$-sheltered subpermutation of $\pi$} and denote by $\pi_S^H$, is $\pi_{i+1}\cdots\pi_{j-1}$. Note that the entry $\pi_j$ does not appear in either of these two parts. This decomposition of $\pi$ into the $H$-unsheltered and $H$-sheltered subpermutations provides a useful decomposition of valid hook configurations in $\VHC^H(\pi)$, which we state in the following lemma. We use Figure \ref{Fig14} as a substitute for the proof of this lemma, leaving the details to the reader.  

\begin{lemma}\label{Lem2}
Let $\pi$ be a permutation with descents $d_1<\cdots<d_k$. If $H$ is a hook of $\pi$ with northeast endpoint $(j,\pi_j)$ and $j>d_k$, then there exists a bijection \[\varphi^H:\VHC^H(\pi)\to\VHC(\pi_U^H)\times\VHC(\pi_S^H).\] 
\end{lemma}

\begin{figure}[h]
\begin{center}
\includegraphics[width=.87\linewidth]{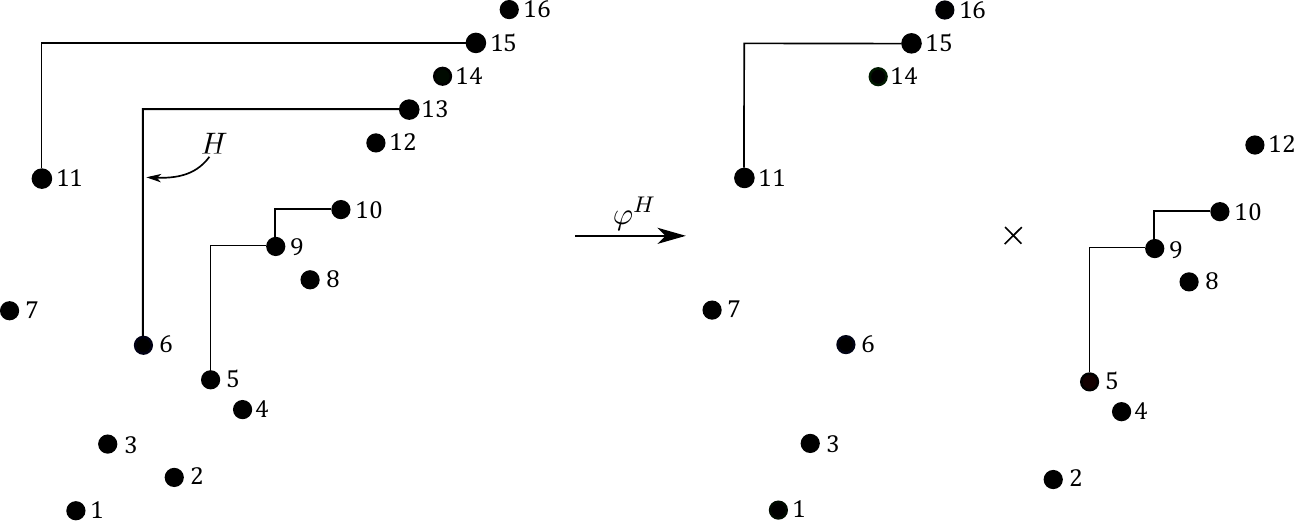}
\caption{The bijection $\varphi^H$ from Lemma \ref{Lem2}}
\label{Fig14}
\end{center}  
\end{figure}

Recall from the introduction that the tail of a permutation $\pi\in S_n$ is the list of points \linebreak $(n-\ell+1,n-\ell+1),\ldots,(n,n)$, where $\ell=\tl(\pi)$ is the tail length of $\pi$. Let $\SW_i(\pi)$ be the set of hooks of a permutation $\pi$ with southwest endpoint $(i,\pi_i)$. We say a descent $d$ of $\pi$ is \emph{tail-bound} if every hook in $\SW_d(\pi)$ has its northeast endpoint in the tail of $\pi$. The following corollary follows immediately from Lemma \ref{Lem2}. It is very closely related to the ``Decomposition Lemma" used to compute fertilities of permutations in \cite{DefantCounting} and \cite{DefantTroupes}.

\begin{corollary}\label{Cor1}
If $d$ is a tail-bound descent of a permutation $\pi\in S_n$, then \[|\VHC(\pi)|=\sum_{H\in\SW_d(\pi)}|\VHC(\pi_U^H)|\cdot|\VHC(\pi_S^H)|.\]
\end{corollary}

\section{$\VHC(\Av(312))$}\label{Sec:312}
This section begins our exploration with an analysis of valid hook configurations of $312$-avoiding permutations. We start by describing a correspondence between $312$-avoiding permutations and certain matrices. It will be convenient to first establish one quick piece of terminology. Given an $\ell\times \ell$ matrix $M=(m_{ij})$ and indices $r,r',c,c'\in\{1,\ldots,\ell\}$, consider the matrix obtained by deleting all rows of $M$ except rows $r$ and $r'$ and deleting all columns of $M$ except columns $c$ and $c'$. We say this new matrix is a \emph{lower $2\times 2$ submatrix} of $M$ if $\ell+1-c\leq r<r'$ and $c<c'$.

Fix $n\geq 1$ and a permutation $\pi=\pi_1\cdots\pi_n\in\Av_n(312)$ such that $\VHC(\pi)$ is nonempty. We must have $\pi_n=n$ (if $\pi_d=n$ with $d<n$, then the point $(d,\pi_d)$ would be a descent top, but it could not be the southwest endpoint of a hook). A \emph{left-to-right maximum of the plot of $\pi$} is a point in the plot of $\pi$ that is higher up than every point to its left. Let $\mathfrak R_0,\ldots,\mathfrak R_\ell$ be these left-to-right maxima listed in order from right to left (so $\mathfrak R_0=(n,n)$ and $\mathfrak R_\ell=(1,\pi_1)$). It will be convenient to let $\mathfrak R_{\ell+1}=(0,0)$, although this is not a point in the plot of $\pi$. Let $M(\pi)=(m_{ij})$ be the $\ell\times \ell$ matrix in which $m_{ij}$ is the number of points in the plot of $\pi$ lying strictly vertically between $\mathfrak R_i$ and $\mathfrak R_{i+1}$ and strictly horizontally between $\mathfrak R_{\ell-j}$ and $\mathfrak R_{\ell-j+1}$. See Figure \ref{Fig5} for an example. 

\begin{figure}[h]
\[
\begin{array}{l}\includegraphics[width=.22\linewidth]{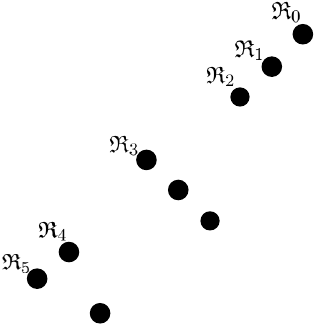}\end{array}\quad M(\pi)=\left( \begin{array}{ccccc}
0 & 0 & 0 & 0 & 0 \\
0 & 0 & 0 & 0 & 0 \\
0 & 0 & 2 & 0 & 0 \\
0 & 0 & 0 & 0 & 0 \\
0 & 1 & 0 & 0 & 0 \end{array} \right)\]
\caption{The matrix $M(\pi)$, where $\pi=231654789$.}
\label{Fig5}
\end{figure} 

\begin{remark}\label{Rem1}
Note that $m_{ij}=0$ whenever $j\leq \ell-i$. Furthermore, in every lower $2\times 2$ submatrix of $M(\pi)$, either the bottom left entry or the top right entry is $0$. Indeed, this follows from the fact that $\pi$ avoids $312$. We can easily reconstruct the plot of the permutation $\pi$ from the matrix $M(\pi)$ by noting that the points lying horizontally between two consecutive left-to-right maxima must be decreasing in height from left to right and that the same must be true of points lying vertically between two left-to-right maxima. \hspace*{\fill}$\lozenge$  
\end{remark}

Let us now choose a valid hook configuration $\mathcal H\in\VHC(\pi)$. Note that every northeast endpoint of a hook in $\mathcal H$ is a left-to-right maximum of the plot of $\pi$. Indeed, this follows from Condition 2 in Definition \ref{Def9} and the fact that $\pi$ avoids $312$. Let $\gamma_i$ be the sum of the entries in column $\ell-i+1$ of $M(\pi)$, and let $\gamma_i'$ be the sum of the entries in row $i$ of $M(\pi)$. Because $m_{ij}=0$ whenever $j\leq \ell-i$, we must have 
\begin{equation}\label{Eq1}
\gamma_1+\cdots+\gamma_p\geq\gamma_1'+\cdots+\gamma_p'\quad\text{for every }p\in\{1,\ldots,\ell\}.
\end{equation}
If $\mathfrak R_{j-1}$ is the northeast endpoint of a hook in $\mathcal H$, let $X_j=U$; otherwise, let $X_j=E$. Let $\Lambda=X_1D^{\gamma_1}X_2 D^{\gamma_2}\cdots X_\ell D^{\gamma_\ell}$ and $\Lambda'=X_1D^{\gamma_1'}X_2 D^{\gamma_2'}\cdots X_\ell D^{\gamma_\ell'}$. Finally, let $\widehat\DL(\mathcal H)=(\Lambda,\Lambda')$. 

We claim that $(\Lambda,\Lambda')\in\Int(\mathcal M_{n-1}^C)$. To see this, let $\Lambda=\Lambda_1\cdots\Lambda_{n-1}$. Recall that the southwest endpoints of the hooks of $\mathcal H$ are precisely the descent tops of the plot of $\pi$. For $p\in\{1,\ldots,\ell\}$, let $Y_p$ be the set of points in the plot of $\pi$ lying to the right of $\mathfrak R_p$. One can easily check that $\gamma_1+\cdots+\gamma_p$ is the number of southwest endpoints of hooks in $\mathcal H$ that are in $Y_p\cup\{\mathfrak R_p\}$. Similarly, the number of $U$'s in $X_1\cdots X_p$ is the number of northeast endpoints of hooks in $\mathcal H$ that are in $Y_p$. Since each southwest endpoints in $Y_p\cup\{\mathfrak R_p\}$ belongs to a hook whose northeast endpoint is in $Y_p$, we see that $\gamma_1+\cdots+\gamma_p$ is at most the number of $U$'s in $X_1\cdots X_p$. This is true for every $p\in\{1,\ldots,\ell\}$, so $\Lambda$ is a Motzkin path. It now follows from \eqref{Eq1} that $\Lambda'$ is also a Motzkin path and that $\Lambda\leq_S\Lambda'$. Finally, an index $i$ appears in the class of $\Lambda$ if and only if $X_i=E$. This occurs if and only if $i$ is in the class of $\Lambda'$. It follows that $\cl(\Lambda)=\cl(\Lambda')$, so $\Lambda\leq_C\Lambda'$. We have now defined a map $\widehat\DL_n:\VHC(\Av_n(312))\to\Int(\mathcal M_{n-1}^C)$. 
\begin{theorem}\label{Thm2}
For each positive integer $n$, the map \[\widehat\DL_n:\VHC(\Av_n(312))\to\Int(\mathcal M_{n-1}^C)\] is a bijection. 
\end{theorem} 

\begin{proof}
The proof relies on the following fact, which appears as Lemma 5.2 in \cite{DefantCatalan}.

\noindent {\bf Fact:} Let $a_1,\ldots,a_\ell,b_1,\ldots,b_\ell$ be nonnegative integers such that $a_1+\cdots+a_\ell=b_1+\cdots+b_\ell$ and $a_{\ell-i+1}+\cdots+a_\ell\leq b_{\ell-i+1}+\cdots+b_\ell$ for all $i\in\{1,\ldots,\ell\}$. There exists an $\ell\times \ell$ matrix $M=(m_{ij})$ with nonnegative integer entries such that 
\begin{enumerate}[(i)]
\item $m_{ij}=0$ whenever $j\leq \ell-i$;
\item the sum of the entries in column $i$ of $M$ is $b_i$ for every $i\in\{1,\ldots,\ell\}$;
\item the sum of the entries in row $i$ of $M$ is $a_{\ell-i+1}$ for every $i\in\{1,\ldots,\ell\}$;
\item in every lower $2\times 2$ submatrix of $M$, either the bottom left entry or the top right entry is $0$.  
\end{enumerate}

To prove that $\widehat\DL_n$ is surjective, fix $(\Lambda,\Lambda')\in\Int(\mathcal M_{n-1}^C)$. Because $\cl(\Lambda)=\cl(\Lambda')$, we can write $\Lambda=X_1D^{\gamma_1}X_2D^{\gamma_2}\cdots X_\ell D^{\gamma_\ell}$ and $\Lambda'=X_1D^{\gamma_1'}X_2D^{\gamma_2'}\cdots X_\ell D^{\gamma_\ell'}$ for some $X_1,\ldots,X_\ell\in\{U,E\}$. Let $a_i=\gamma_{\ell-i+1}'$ and $b_i=\gamma_{\ell-i+1}$. Because $\Lambda$ and $\Lambda'$ are Motzkin paths, we have $a_1+\cdots+a_\ell=b_1+\cdots+b_\ell$. The fact that $\Lambda\leq_S\Lambda'$ tells us that $a_{\ell-i+1}+\cdots+a_\ell\leq b_{\ell-i+1}+\cdots+b_\ell$ for all $i\in\{1,\ldots,\ell\}$. The above fact guarantees that there is a matrix $M=(m_{ij})$ satisfying the properties (i)--(iv). 
According to Remark \ref{Rem1}, we can use such a matrix to obtain a permutation $\pi\in \Av_n(312)$ such that $M(\pi)=M$ and $\pi_n=n$. We claim that there is a unique valid hook configuration $\mathcal H\in\VHC(\pi)$ such that $\widehat\DL_n(\mathcal H)=(\Lambda,\Lambda')$. Let $\mathfrak R_0,\ldots,\mathfrak R_\ell$ be the left-to-right maxima of the plot of $\pi$ listed in order from right to left, and let $A=\{\mathfrak R_i:X_{i+1}=U\}$. The northeast endpoints of the hooks in $\mathcal H$ are precisely the points in $A$. The specific choices of the hooks themselves are now determined by the conditions in Definition \ref{Def9}. We need to make sure that each descent top of the plot of $\pi$ can actually find a corresponding northeast endpoint for its hook. In other words, we need to know that for each $i\in\{1,\ldots,n-1\}$, the number of descent tops in the set $\{(n-i,\pi_{n-i}),\ldots,(n,\pi_n)\}$ is at most the number of elements of $A$ in $\{(n-i+1,\pi_{n-i+1}),\ldots,(n,\pi_n)\}$. This follows immediately from the fact that $\Lambda$ is a Motzkin path. 

To prove injectivity, let us assume there exist $\pi,\pi'\in\Av_n(312)$, $\mathcal H\in\VHC(\pi)$, and $\mathcal H'\in\VHC(\pi')$ with $\widehat\DL_n(\mathcal H)=\widehat\DL_n(\mathcal H')=(\Lambda,\Lambda')$, where $(\Lambda,\Lambda')$ is as above. Assume by way of contradiction that $\mathcal H\neq\mathcal H'$. We saw in the preceding paragraph that $\mathcal H$ is uniquely determined by $\pi$ and $(\Lambda,\Lambda')$. This means we must have $\pi\neq\pi'$. According to Remark \ref{Rem1}, the matrices $M(\pi)=(m_{ij})$ and $M(\pi')=(m_{ij}')$ uniquely determine $\pi$ and $\pi'$, respectively. Therefore, these matrices are distinct. However, both of these matrices satisfy properties (i)--(iv) from the above fact, where $a_i=\gamma_{k-i+1}'$ and $b_i=\gamma_{k-i+1}$. Because they are distinct, we can find a pair $(i_0,j_0)$ with $m_{i_0j_0}\neq m_{i_0j_0}'$. We may assume that $j_0$ was chosen maximally, which means $m_{ij}=m_{ij}'$ whenever $j>j_0$. We may assume that $i_0$ was chosen maximally after $j_0$ was chosen, meaning $m_{ij_0}=m_{ij_0}'$ whenever $i>i_0$. We may assume without loss of generality that $m_{i_0j_0}>m_{i_0j_0}'$. Because $M(\pi)$ and $M(\pi')$ satisfy property (ii), their $j_0^\text{th}$ columns have the same sum. This means that there exists $i_1\neq i_0$ with $m_{i_1j_0}<m_{i_1j_0}'$. In particular, $m_{i_1j_0}'$ is positive. The maximality of $i_0$ guarantees that $i_1<i_0$. Because $M(\pi)$ and $M(\pi')$ satisfy property (iii), their $i_1^\text{th}$ rows have the same sum. This means that there exists $j_1\neq j_0$ with $m_{i_1j_1}>m_{i_1j_0}'$. The maximality of $j_0$ guarantees that $j_1<j_0$. Since $M(\pi)$ satisfies property (i) and $m_{i_1j_1}>0$, we must have $\ell+1-j_1\leq i_1$. Now, the $j_1^\text{th}$ columns of $M(\pi)$ and $M(\pi')$ have the same sum, so there exists $i_2\neq i_1$ such that $m_{i_2j_1}<m_{i_2j_1}'$. If $i_2>i_1$, then $m_{i_2j_1}'$ and $m_{i_1j_0}'$ are positive numbers that form the bottom left and top right entries in a lower $2\times 2$ submatrix of $M(\pi)$. This is impossible since $M(\pi)$ satisfies property (iv), so we must have $i_2<i_1$. Continuing in this fashion, we find decreasing sequences of positive integers $i_0>i_1>i_2>\cdots$ and $j_0>j_1>j_2>\cdots$. This is our desired contradiction. 
\end{proof}

The preceding theorem tells us that $|\VHC(\Av_n(312))|=|\Int(\mathcal M_{n-1}^C)|$. We now turn our attention toward obtaining a recurrence for these numbers. If $\lambda=\lambda_1\cdots\lambda_\ell\in S_\ell$ and $\mu=\mu_1\ldots\mu_m\in S_m$, then the \emph{direct sum} of $\lambda$ and $\mu$, denoted $\lambda\oplus\mu$, is the permutation in $S_{\ell+m}$ obtained by ``placing $\mu$ above and to the right of $\lambda$." More formally, the $i^\text{th}$ entry of $\lambda\oplus\mu$ is \[(\lambda\oplus\mu)_i=\begin{cases} \lambda_i & \mbox{if } 1\leq i\leq \ell; \\ \mu_{i-\ell}+\ell & \mbox{if } \ell+1\leq i\leq \ell+m. \end{cases}\] A permutation is called \emph{sum indecomposable} if it cannot be written as a direct sum of two shorter permutations. Every normalized permutation $\pi$ can be written uniquely in the form $\pi=\lambda^{(1)}\oplus\cdots\oplus\lambda^{(r)}$ for some sum indecomposable permutations $\lambda^{(1)},\ldots,\lambda^{(r)}$, which are called the \emph{components} of $\pi$. Let $\comp(\pi)$ denote the number of components of $\pi$. Recall the notation from Corollary \ref{Cor1}.  

Let \[\mathcal D_{\ell,c}(n)=\{\pi\in\Av_{n+\ell}(312):\tl(\pi)=\ell,\comp(\pi)=c\}\] and \[\mathcal D_{\geq\ell,\geq c}(n)=\{\pi\in\Av_{n+\ell}(312):\tl(\pi)\geq\ell,\comp(\pi)\geq c\}.\] Let $B_{\ell,c}(n)=|\VHC(\mathcal D_{\ell,c}(n))|$ and $B_{\geq\ell,\geq c}(n)=|\VHC(\mathcal D_{\geq\ell}(n))|$. 

Suppose $\pi\in\mathcal D_{\ell,c}(n+1)$ for some $n\geq 0$ and $c\geq\ell+1$. Note that $\pi$ is not an identity permutation because it has length $\ell+n+1$ and has tail length $\ell$. Because $\pi$ avoids $312$, one can easily check that $n$ is a tail-bound descent of $\pi$. Choose $j\in\{1,\ldots,\ell\}$, and let $H$ be the hook of $\pi$ with southwest endpoint $(n,\pi_n)$ and northeast endpoint $(n+1+j,n+1+j)$. We have $\pi_U^H=\pi_1\cdots\pi_{n-1}(n+2+j)\cdots(n+\ell+1)$ and $\pi_S^H=\pi_{n+1}(n+2)\cdots(n+j)$. The increasing permutation $\pi_S^H$ has a unique valid hook configuration (with no hooks). The permutation $\pi_U^H$ has the same number of valid hook configurations as its normalization. This normalization is an element of $\mathcal D_{\geq\ell-j,\geq c-j}(n)$. Any element of $\mathcal D_{\geq\ell-j,\geq c-j}(n)$ can be the normalization of $\pi_U^H$. Combining these facts with Corollary \ref{Cor1}, we find that 
\begin{equation}\label{Eq7}
B_{\ell,c}(n+1)=\sum_{j=1}^\ell B_{\geq\ell-j,\geq c-j}(n).
\end{equation}  
Figure \ref{Fig6} illustrates this recurrence. 

\begin{figure}[h]
\begin{center}
\includegraphics[width=.42\linewidth]{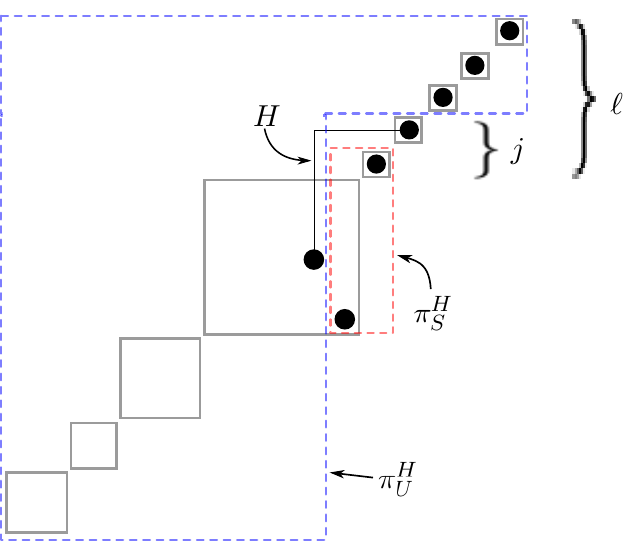}
\caption{An illustration of the recurrence in equation \eqref{Eq7}. This is an example with $c=9$, where the grey boxes represent the components of the permutation $\pi$.}
\label{Fig6}
\end{center}  
\end{figure}

If $c\geq \ell+1$, then we also have $B_{\ell,c}(0)=0$. If $c\leq\ell-1$, then $B_{\ell,c}(n)=0$. Finally, $B_{\ell,\ell}(0)=1$ and $B_{\ell,\ell}(n)=0$ for $n\neq 0$. These initial conditions and the recurrence in \eqref{Eq2} allow us to efficiently compute the values of $B_{\ell,c}(n)$. Hence, we can efficiently compute the numbers $|\VHC(\Av_n(312))|=B_{\geq 0,\geq 0}(n)$. The first few values, starting at $n=1$, are \[1, 1, 2, 5, 14, 44, 148, 528, 1972, 7647, 30605, 125801, 529131, 2270481, 9914870, 43973755, 197744417.\] We can also use this recurrence to derive the generating function equation in the following proposition; we omit the details. 

\begin{proposition}
We have \[\sum_{n\geq 0}|\VHC(\Av_n(312))|x^n=1+\sum_{n\geq 1}|\Int(\mathcal M_{n-1}^C)|x^n=Q(x,0,0),\] where $Q(x,y,z)$ is the trivariate power series satisfying \[\frac{y}{1-y}(Q(x,y,z)-Q(x,y,0))=\frac{Q(x,y,z)-1/(1-y)}{x}-\frac{Q(x,y,z)-Q(x,0,z)}{y}(z-1)\] \[-\frac{Q(x,y,0)-Q(x,0,0)}{y}-\frac{Q(x,y,z)-Q(x,y,0)}{xz}.\]
\end{proposition}

Let $\mathfrak w(k)$ be the number of lattice walks of length $k$ that start and end at the origin, always stay in the first quadrant, and use the steps $(-1, 0), (-1, 1), (0, -1), (0, 1), (1, -1)$. These walks were studied in \cite{Bostan} and \cite{Bousquet5}. The first few values of $\mathfrak w(k)$ are given in the OEIS sequence A151347. The following theorem, which links valid hook configurations in $\VHC(\Av_n(312))$, intervals in the posets $\mathcal M_{n-1}^C$, and these specific lattice walks, was stated as a conjecture in the original preprint version of this article. It was subsequently proven by Maya Sankar.  

\begin{theorem}[\!\!\cite{Maya}]
For every positive integer $n$, we have
\[|\VHC(\Av_n(312))|=|\Int(\mathcal M_{n-1}^C)|=\sum_{k=0}^{n-1}{n-1\choose k}\mathfrak w(k).\]
\end{theorem}  

\section{$\VHC(\Av(132))$ and $\VHC(\Av(231))$}\label{Sec:132}

We now consider the pattern $132$ and the pattern $231$. Our goal in this section is to prove the following theorem. 

\begin{theorem}\label{Thm1}
Let $\rho\approx 4.658905$ be the unique real root of $256x^3-645x^2-2112x-2048$, and let $\beta\approx 0.805810$ be the unique positive real root of $41472x^6-34749x^4+5472 x^2-256$. We have \[\sum_{n\geq 0}|\VHC(\Av_n(132))|x^n=\sum_{n\geq 0}|\VHC(\Av_n(231))|x^n=1+\sum_{n\geq 1}|\Int(\mathcal M_{n-1}^T)|x^n,\] and this generating function is algebraic of degree $5$. Furthermore, \[|\VHC(\Av_n(132))|=|\VHC(\Av_n(231))|\sim \frac{\beta}{\sqrt{\pi}}\cdot \frac{\rho^n}{n^{5/2}}.\] 
\end{theorem}

The proof of this theorem is similar to the proof of the formula for $|s^{-1}(\Av_n(231))|$ (which is the number of $2$-stack-sortable permutations in $S_n$) given in \cite{DefantCounting}. In fact, the recurrence relations obtained in that proof and in the following argument are identical except in their initial conditions. The proof in \cite{DefantCounting} uses the sequence of Catalan numbers for initial values, whereas we use the simpler sequence $1,1,1,\ldots$ here. What is interesting is that the generating function for $|\VHC(\Av_n(231))|$ is actually more complicated than the generating function for $|s^{-1}(\Av_n(231))|$. Both are algebraic, but their degrees are $5$ and $3$, respectively. Moreover, the radius of convergence of the former is a cubic irrational while the radius of convergence of the latter is simply $4/27$.  

\begin{proof}
In \cite{DefantFertilityWilf}, the author found a bijection $\VHC(\Av(132))\to\VHC(\Av(231))$ that preserves the lengths of the underlying permutations (and much more). Therefore,
$\sum_{n\geq 0}|\VHC(\Av_n(132))|x^n=\sum_{n\geq 0}|\VHC(\Av_n(231))|x^n$. Let \[\mathcal D_\ell(n)=\{\pi\in\Av_{n+\ell}(231):\tl(\pi)=\ell\}\quad\text{and}\quad \mathcal D_{\geq\ell}(n)=\{\pi\in\Av_{n+\ell}(231):\tl(\pi)\geq\ell\}.\] Let $B_\ell(n)=|\VHC(\mathcal D_\ell(n))|$ and $B_{\geq\ell}(n)=|\VHC(\mathcal D_{\geq\ell}(n))|$. 

Suppose $\pi\in\mathcal D_\ell(n+1)$ is such that $\pi_{n+1-i}=n+1$ (where $n\geq 0$). Then $n+1-i$ is a tail-bound descent of $\pi$. 
According to Corollary \ref{Cor1}, $|\VHC(\pi)|$ is equal to the number of triples $(H,\mathcal H_U,\mathcal H_S)$, where $H\in\SW_{n+1-i}(\pi)$, $\mathcal H_U\in \VHC(\pi_U^H)$, and $\mathcal H_S\in \VHC(\pi_S^H)$. Choosing $H$ amounts to choosing the number $j\in\{1,\ldots,\ell\}$ such that the northeast endpoint of $H$ is $(n+1+j,n+1+j)$. The permutation $\pi$ and the choice of $H$ determine the permutations $\pi_U^H$ and $\pi_S^H$. On the other hand, the choices of $H$ and the permutations $\pi_U^H$ and $\pi_S^H$ uniquely determine $\pi$. It follows that $B_\ell(n+1)$, which is the number of ways to choose an element of $\VHC(\mathcal D_\ell(n+1))$, is also the number of ways to choose $j$, the permutations $\pi_U^H$ and $\pi_S^H$, and the valid hook configurations $\mathcal H_U$ and $\mathcal H_S$. Let us fix a choice of $j$. 

Because $\pi$ avoids $231$, $\pi_U^H$ must be a permutation of the set $\{1,\ldots,n-i\}\cup\{n+1\}\cup\linebreak\{n+2+j,\ldots,n+\ell+1\}$, while $\pi_S^H$ must be a permutation of $\{n-i+1,\ldots,n+j\}\setminus\{n+1\}$. Therefore, choosing $\pi_U^H$ and $\pi_S^H$ is equivalent to choosing their normalizations. The normalization of $\pi_U^H$ is in $\mathcal D_{\geq \ell-j+1}(n-i)$, while the normalization of $\pi_S^H$ is in $\mathcal D_{\geq j-1}(i)$. Any element of $\mathcal D_{\geq \ell-j+1}(n-i)$ can be chosen as the normalization of $\pi_U^H$, and any element of $\mathcal D_{\geq j-1}(i)$ can be chosen as the normalization of $\pi_S^H$. Also, every permutation has the same number of valid hook configurations as its normalization. Combining these facts, we find that the number of choices for $\pi_U^H$ and $\mathcal H_U$ is $|\VHC(\mathcal D_{\geq \ell-j+1}(n-i))|=B_{\geq \ell-j+1}(n-i)$. Similarly, the number of choices for $\pi_S^H$ and $\mathcal H_S$ is $B_{\geq j-1}(i)$. Consequently, 
\begin{equation}\label{Eq3}
B_\ell(n+1)=\sum_{i=1}^n\sum_{j=1}^\ell B_{\geq \ell-j+1}(n-i)B_{\geq j-1}(i).
\end{equation}

Let \[G_\ell(x)=\sum_{n\geq 0}B_{\geq\ell}(n)x^n\quad\text{and}\quad I(x,z)=\sum_{\ell\geq 0}G_\ell(x)z^\ell.\] Note that \[G_\ell(0)=B_{\geq \ell}(0)=|\VHC(\mathcal D_{\geq \ell}(0))|=|\VHC(123\cdots\ell)|=1.\] 
Because $B_{\geq 0}(n)=|\VHC(\Av_n(231))|$, we wish to understand the generating function \[I(x,0)=G_0(x)=\sum_{n\geq 0}B_{\geq 0}(n)x^n=\sum_{n\geq 0}|\VHC(\Av_n(231))|x^n.\] 

By \eqref{Eq3}, we have \[\sum_{\ell\geq 0}\sum_{n\geq 0}B_\ell(n+1)x^nz^\ell=\sum_{\ell\geq 0}\sum_{j=1}^\ell \sum_{n\geq 0}\sum_{i=1}^n B_{\geq \ell-j+1}(n-i)B_{\geq j-1}(i)x^nz^\ell\] \[=\sum_{\ell\geq 0}\sum_{j=1}^\ell G_{\ell-j+1}(x)(G_{j-1}(x)-G_{j-1}(0))z^\ell=\sum_{\ell\geq 0}\sum_{j=1}^\ell G_{\ell-j+1}(x)(G_{j-1}(x)-1)z^\ell\] 
\begin{equation}\label{Eq4}
=\left(\sum_{r\geq 0}G_{r+1}(x)z^r\right)\left(\sum_{j\geq 1}(G_{j-1}(x)-1)z^j\right)=(I(x,z)-I(x,0))(I(x,z)-1/(1-z)).
\end{equation}

On the other hand, \[B_\ell(n+1)=B_{\geq\ell}(n+1)-B_{\geq \ell+1}(n),\] so \[\sum_{\ell\geq 0}\sum_{n\geq 0}B_\ell(n+1)x^nz^\ell=\sum_{\ell\geq 0}\sum_{n\geq 0}B_{\geq\ell}(n+1)x^nz^\ell-\sum_{\ell\geq 0}\sum_{n\geq 0}B_{\geq \ell+1}(n)x^nz^\ell\]
\begin{equation}\label{Eq5}
=\frac{1}{x}\sum_{\ell\geq 0}(G_\ell(x)-1)z^\ell-\frac{1}{z}\sum_{\ell\geq 0}G_{\ell+1}(x)z^{\ell+1}=\frac{I(x,z)-1/(1-z)}{x}-\frac{I(x,z)-I(x,0)}{z}.
\end{equation} 
By \eqref{Eq4} and \eqref{Eq5}, we have 
\begin{equation}\label{Eq6}
xz(I(x,z)-I(x,0))(I(x,z)(1-z)-1)-z(I(x,z)(1-z)-1)+x(1-z)(I(x,z)-I(x,0))=0.
\end{equation}

At this point, we use the ``quadratic method," which is discussed and substantially generalized in \cite{Bousquet3}. We can rewrite \eqref{Eq6} as 
\begin{equation}\label{Eq2}
R(I(x,z),I(x,0),x,z)^2=\Delta(I(x,0),x,z),
\end{equation} where
\[R(u,v,x,z)=2xz(1-z)u+x-z-2xz+z^2-xz(1-z)v\] and \[\Delta(v,x,z)=(x - (2 + v) x z + (-1 + z) z + v x z^2)^2-4xz(1-z)(z + v x (2 z-1)).\] There is a unique power series $Z=Z(x)$ such that $Z(x)=x+O(x^2)$ and $R(I(x,Z),I(x,0),x,Z)$ $=0$. According to \eqref{Eq2}, $z=Z(x)$ is a repeated root of $\Delta(I(x,0),x,z)$. This means that the discriminant of $\Delta(I(x,0),x,z)$ with respect to $z$ vanishes. Computing this discriminant, we find that $Q(I(x,0),x)=0$, where 
\[Q(v,x)=(-1 + 6 x + 15 x^2 + 
   8 x^3) + (1 - 11 x + 28 x^3 + 16 x^4) v + (4 x - 19 x^2 - 
    14 x^3) v^2 \] 
\begin{equation}\label{Eq8}
+ (6 x^2 - 9 x^3 + 8 x^4) v^3 + 4 x^3 v^4 + x^4 v^5.
\end{equation}     
This polynomial is irreducible, so the generating function $I(x,0)=\sum_{n\geq 0}|\VHC(\Av_n(231))|x^n=\sum_{n\geq 0}|\VHC(\Av_n(132))|x^n$ is algebraic of degree $5$ over the field $\mathbb R(x)$. Fang \cite{Fang2} found an algebraic equation (written a bit differently) that the generating function $1+\sum_{n\geq 1}|\Int(\mathcal M_{n-1}^T)|x^n$ satisfies, and one can show (after some straightforward manipulations that we omit) that it matches the equation we have just found for $I(x,0)$. 

We are left to prove the desired asymptotic formula for the coefficients of $I(x,0)$. To do this, we invoke the techniques of singularity analysis outlined in Chapters VI and VII of \cite{Flajolet}. We refer the reader to that book for the relevant definitions and details. The singularities of $I(x,0)$ are contained in the set of roots of $x^{19} (1 + x)^2 (2048 x^3 + 2112 x^2+ 645 x -256)^3$, which is the discriminant of $Q(v,x)$ with respect to $v$. Pringsheim's theorem guarantees that the radius of convergence of $I(x,0)$ is one of these singularities. It follows that the radius of convergence must be $1/\rho$, where $\rho$ is as in the statement of the theorem. It is now routine to identify the branch of $Q(v,x)$ near $v=1/\rho$ that corresponds to the combinatorially-defined generating function $I(x,0)$ and expand it as a Puisseux series. We find that \[I(x,0)=\alpha_1+\alpha_2(1/\rho-x)+\alpha_3(1/\rho-x)^{3/2}+o((1/\rho-x)^{3/2})\] for some algebraic numbers $\alpha_1,\alpha_2,\alpha_3$. Moreover, $\alpha_3=4\beta\rho^{3/2}/3$, where $\beta$ is as in the statement of the theorem. This translates to the asymptotic formula \[|\VHC(\Av_n(231))|\sim\alpha_3\rho^{n-3/2}\frac{\Gamma(n-3/2)}{\Gamma(-3/2)\Gamma(n+1)}\sim \frac{\beta}{\sqrt{\pi}}\cdot\frac{\rho^n}{n^{5/2}}. \qedhere\]
\end{proof}

\begin{remark}\label{Rem2}
We have shown that $|\VHC(\Av_n(132))|=|\Int(\mathcal M_{n-1}^T)|$ for each $n\geq 1$. In the original preprint version of this article, we asked for a direct combinatorial proof of this fact. Such a proof was found soon afterward by Maya Sankar \cite{Maya}. 
\end{remark}

\section{$\VHC(\Av(132,231))$, $\VHC(\Av(132,312))$, and $\VHC(\Av(231,312))$}\label{Sec:Pairs}
 
This brief section is dedicated to proving the following theorem. Recall that $\mathcal M_n^A$ is the antichain on the set ${\bf M}_n$ and that $M_n=|{\bf M}_n|=|\Int(\mathcal M_n^A)|$ is the $n^\text{th}$ Motzkin number. 

\begin{theorem}\label{Thm3}
For every positive integer $n$, the bijection $\widehat\DL_n:\VHC(\Av_n(312))\to\Int(\mathcal M_{n-1}^C)$ restricts to a bijection $\VHC(\Av_n(231,312))\to\Int(\mathcal M_{n-1}^A)$. Furthermore, \[|\VHC(\Av_n(132,231))|=|\VHC(\Av_n(132,312))|=|\VHC(\Av_n(231,312))|=M_{n-1}.\]
\end{theorem} 
 
\begin{proof}
Recall the definition of the components of a permutation from Section \ref{Sec:312}. A permutation is called \emph{layered} if its components are decreasing permutations. For example, the permutation $2143765=(21)\oplus(21)\oplus(321)$ is layered because its components are $21$, $21$, and $321$. It is well known that the set of layered permutations in $S_n$ is precisely $\Av_n(231,312)$. If $\pi\in\Av_n(312)$, $\mathcal H\in\VHC(\pi)$, and $\widehat\DL_n(\mathcal H)=(\Lambda,\Lambda')$, then it is straightforward to check that $\Lambda=\Lambda'$ if and only if $\pi$ is layered. In other words, $\DL_n(\mathcal H)\in\Int(\mathcal M_{n-1}^A)$ if and only if $\mathcal H\in\VHC(\Av_n(231,312))$. In \cite{DefantFertilityWilf}, the author found bijections $\VHC(\Av(231,312))\to\VHC(\Av(132,312))$ and $\VHC(\Av(132,231))\to\VHC(\Av(231,312))$ that preserve the lengths of the underlying permutations. This proves the last statement of the theorem. 
\end{proof}
 
\section{$\VHC(\Av(132,321))$}\label{Sec:132,321}

\begin{theorem}\label{Thm4}
We have \[\sum_{n\geq 0}|\VHC(\Av_n(132,321))|x^n=\frac{1-3x+3x^2}{(1-x)^4}.\] 
\end{theorem}
\begin{proof}
Choose $n\geq 1$ and $\ell\in\{1,\ldots,n-1\}$. By concatenating the increasing permutations \linebreak $(i+1)\cdots(n-\ell)$, $1\cdots i$, and $(n-\ell+1)\cdots n$, we obtain the permutation $\zeta_i=(i+1)\cdots$\linebreak$(n-\ell)1\cdots i(n-\ell+1)\cdots n$. One can easily check that $\zeta_1,\ldots,\zeta_{n-\ell-1}$ are precisely the permutations in $\Av_n(132,321)$ with tail length $\ell$ and that each of these permutations has exactly $\ell$ valid hook configurations (each of which has exactly $1$ hook). It follows that \[|\VHC_n(132,321)|=1+\sum_{\ell=1}^{n-1}(n-\ell-1)\ell.\] The remainder of the proof is routine.  
 \end{proof}
 
\section{$\VHC(231,321)$}\label{Sec:231,321}

\begin{theorem}\label{Thm5}
We have \[\sum_{n\geq 0}|\VHC(\Av_n(231,321))|x^n=\frac{1-2x+2x^2-\sqrt{1-4x+4x^2-4x^3+4x^4}}{2x^2}.\] 
\end{theorem}
\begin{proof}
Let \[\mathcal D_\ell(n)=\{\pi\in\Av_{n+\ell}(231,321):\tl(\pi)=\ell\}\quad\text{and}\quad \mathcal D_{\geq\ell}(n)=\{\pi\in\Av_{n+\ell}(231,321):\tl(\pi)\geq\ell\}.\] Let $B_\ell(n)=|\VHC(\mathcal D_\ell(n))|$ and $B_{\geq\ell}(n)=|\VHC(\mathcal D_{\geq\ell}(n))|$. 

Suppose $\pi\in\mathcal D_\ell(n+1)$ is such that $\pi_{n+1-i}=n+1$ (where $n\geq 0$). Then $n+1-i$ is a tail-bound descent of $\pi$. 
Corollary \ref{Cor1} tells us that $|\VHC(\pi)|$ is equal to the number of triples $(H,\mathcal H_U,\mathcal H_S)$, where $H\in\SW_{n+1-i}(\pi)$, $\mathcal H_U\in \VHC(\pi_U^H)$, and $\mathcal H_S\in \VHC(\pi_S^H)$. Choosing $H$ amounts to choosing the number $j\in\{1,\ldots,\ell\}$ such that the northeast endpoint of $H$ is $(n+1+j,n+1+j)$. As in the proof of Theorem \ref{Thm1}, we find that $B_\ell(n+1)$ is the number of ways to choose $j$, the permutations $\pi_U^H$ and $\pi_S^H$, and the valid hook configurations $\mathcal H_U$ and $\mathcal H_S$. Fix a choice of $j$. 

Because $\pi$ avoids $231$, $\pi_U^H$ must be a permutation of the set $\{1,\ldots,n-i\}\cup\{n+1\}\cup\linebreak\{n+2+j,\ldots,n+\ell+1\}$, while $\pi_S^H$ must be a permutation of $\{n-i+1,\ldots,n+j\}\setminus\{n+1\}$. Because $\pi$ avoids $321$, $\pi_S^H$ is the increasing permutation on the set $\{n-i+1,\ldots,n+j\}\setminus\{n+1\}$. There is one choice for $\pi_S^H$ and $\mathcal H_S$. Choosing $\pi_U^H$ is equivalent to choosing its normalization, which is in $\mathcal D_{\geq \ell-j+1}(n-i)$. Any element of $\mathcal D_{\geq \ell-j+1}(n-i)$ can be chosen as the normalization of $\pi_U^H$. Furthermore, every permutation has the same number of valid hook configurations as its normalization. Combining these facts, we find that the number of choices for $\pi_U^H$ and $\mathcal H_U$ is $|\VHC(\mathcal D_{\geq \ell-j+1}(n-i))|=B_{\geq \ell-j+1}(n-i)$. Thus, 
\begin{equation}\label{Eq9}
B_\ell(n+1)=\sum_{i=1}^n\sum_{j=1}^\ell B_{\geq \ell-j+1}(n-i).
\end{equation}

Let \[G_\ell(x)=\sum_{n\geq 0}B_{\geq\ell}(n)x^n\quad\text{and}\quad I(x,z)=\sum_{\ell\geq 0}G_\ell(x)z^\ell.\] Note that $G_\ell(0)=B_{\geq \ell}(0)=|\VHC(\mathcal D_{\geq \ell}(0))|=|\VHC(123\cdots\ell)|=1$. We wish to understand the generating function \[I(x,0)=G_0(x)=\sum_{n\geq 0}B_{\geq 0}(n)x^n=\sum_{n\geq 0}|\VHC(\Av_n(231,321))|x^n.\]

By \eqref{Eq9}, we have \[\sum_{\ell\geq 0}\sum_{n\geq 0}B_{\geq\ell}(n+1)x^nz^\ell=\sum_{\ell\geq 0}\sum_{n\geq 0}\sum_{i=1}^n\sum_{j=1}^\ell B_{\geq \ell-j+1}(n-i)x^nz^\ell=\sum_{\ell\geq 0}\sum_{j=1}^\ell \frac{x}{1-x}G_{\ell-j+1}(x)z^\ell\] 
\begin{equation}\label{Eq10}
=\frac{x(I(x,z)-I(x,0))}{(1-x)(1-z)}.
\end{equation} 
The same argument used to deduce \eqref{Eq5} in the proof of Theorem \ref{Thm1} shows that \begin{equation}\label{Eq11}
\sum_{\ell\geq 0}\sum_{n\geq 0}B_{\geq\ell}(n+1)x^nz^\ell=\frac{I(x,z)-1/(1-z)}{x}-\frac{I(x,z)-I(x,0)}{z}.
\end{equation}  
We can combine \eqref{Eq10} and \eqref{Eq11} and rearrange terms to obtain the equation 
\begin{equation}\label{Eq12}
I(x,z)((1-x)z^2+(-1+2x^2)z+x-x^2)=(-1 + x) z + I(x,0)(x-x^2+(-x+2x^2)z).
\end{equation} 

We now use the \emph{kernel method} (see \cite{Bousquet4,Bousquet3,Banderier,
Prodinger} for details about this method). Let $Z(x)=\dfrac{1-2x^2-\sqrt{1 - 4 x + 4 x^2 - 4 x^3 + 4 x^4}}{2(1-x)}$ so that $(1-x)Z(x)^2+(-1+2x^2)Z(x)+x-x^2=0$. We can substitute $z=Z(x)$ in \eqref{Eq12} to find that $(-1+x)Z(x)+I(x,0)(x-x^2+(-x+2x^2)Z(x))=0$. Thus, \[I(x,0)=\frac{(1-x)Z(x)}{x-x^2+(-x+2x^2)Z(x)}=\frac{(1-x)Z(x)}{-((1-x)Z(x)^2+(-1+2x^2)Z(x))+(-x+2x^2)Z(x)}\] \[=\frac{1}{1-Z(x)}=\frac{1-2x+2x^2-\sqrt{1-4x+4x^2-4x^3+4x^4}}{2x^2}. \qedhere\]
\end{proof} 

\section{$\VHC(\Av(312,321))$}\label{Sec:312,321}

In the proof of the following theorem, it will be helpful to consider a new statistic defined on valid hook configurations. Suppose $\mathcal H$ is a valid hook configuration of a permutation $\pi$. Recall that a left-to-right maximum of the plot of $\pi$ is a point in the plot of $\pi$ that is higher than every point to its left. An \emph{active site} of $\mathcal H$ is a left-to-right maximum of the plot of $\pi$ that is not a northeast endpoint of a hook in $\mathcal H$. Define the \emph{activity} of $\mathcal H$, denoted $\act(\mathcal H)$, to be the number of active sites of $\mathcal H$. For example, the activity of the valid hook configuration \[\includegraphics[width=.14\linewidth]{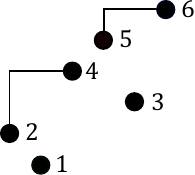}
\] is $2$ because the active sites are $(1,2)$ and $(4,5)$. 

\begin{theorem}\label{Thm6}
For each positive integer $n$, we have \[|\VHC(\Av_n(312,321))|=\sum_{k=0}^{\left\lfloor\frac{n-1}{2}\right\rfloor}\frac{1}{2k+1}{n-k-1\choose k}{n\choose 2k}.\] 
\end{theorem}

\begin{proof}
Let \[\mathcal E_a(n)=\{\mathcal H\in\VHC(\Av_{n+a}(312,321)):\act(\mathcal H)=a\}\] and \[\mathcal E_{\geq a}(n)=\{\mathcal H\in\VHC(\Av_{n+a}(312,321)):\act(\mathcal H)\geq a\}.\] Let $A_a(n)=|\mathcal E_a(n)|$ and $A_{\geq a}(n)=|\mathcal E_{\geq a}(n)|$. Let $\widetilde{\mathcal E}_a(n)$ be the set of valid hook configurations $\mathcal H\in\mathcal E_a(n)$ such that the first entry of the underlying permutation of $\mathcal H$ is not $1$. Removing the leftmost point (which is also the point with height $1$) from each valid hook configuration in $\mathcal E_a(n+1)\setminus \widetilde{\mathcal E}_a(n+1)$ yields a bijection from $\mathcal E_a(n+1)\setminus \widetilde{\mathcal E}_a(n+1)$ to $\mathcal E_{a-1}(n+1)$, so $A_a(n+1)=A_{a-1}(n+1)+|\widetilde{\mathcal E}_a(n+1)|$. 

Now suppose $\mathcal H\in\widetilde{\mathcal E}_a(n+1)$, and let $\pi=\pi_1\cdots\pi_{n+a+1}$ be the underlying permutation of $\mathcal H$. Let $r\geq 1$ be such that $\pi_{r+1}=1$. Because $\pi$ avoids $312$ and $321$, we must have $\pi_1\cdots\pi_r=23\cdots(r+1)$. The points $(i,i+1)$ for $i\in\{1,\ldots,r\}$ are all active sites of $\mathcal H$, so $r\leq a$. There must be a hook of $\mathcal H$ with southwest endpoint $(r,r+1)$. If we remove this hook along with all of the points $(i,\pi_i)$ for $i\in\{1,\ldots,r+1\}$ and then ``normalize" the remaining points and hooks, we obtain a valid hook configuration $\mathcal H'\in\mathcal E_{\geq a-r+1}(n-1)$. On the other hand, it is easy to recover $\mathcal H$ if we are just given $\mathcal H'$ and the values of $n,a,r$. We depict this in Figure \ref{Fig1}. It follows that \[|\widetilde{\mathcal E}_a(n+1)|=\sum_{r=1}^a|\mathcal E_{\geq a-r+1}(n-1)|=\sum_{r=1}^aA_{\geq a-r+1}(n-1).\] Consequently, 
\begin{equation}\label{Eq13}
A_a(n+1)=A_{a-1}(n+1)+\sum_{r=1}^aA_{\geq a-r+1}(n-1).
\end{equation}

\begin{figure}[h]
\begin{center}
\includegraphics[width=.54\linewidth]{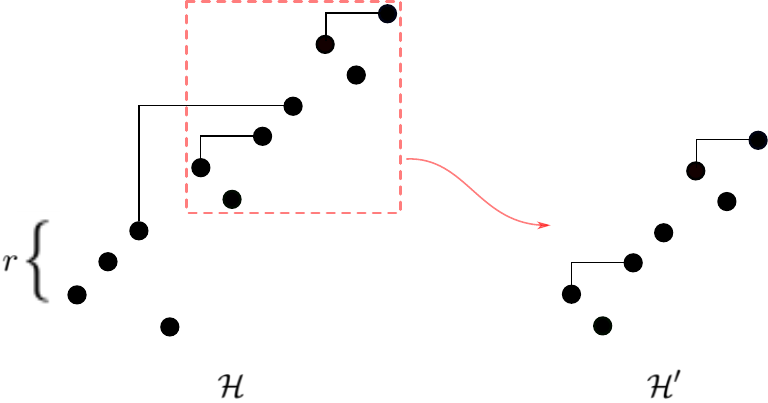}
\caption{An illustration of the proof of Theorem \ref{Thm6}. In this example, $n=6$, $a=4$, and $r=3$.}
\label{Fig1}
\end{center}  
\end{figure}
 
Now let $J(x,z)=\sum_{a\geq 0}\sum_{n\geq 0}A_{\geq a}(n)x^nz^a$. Note that we are primarily interested in the generating function \[J(x,0)=\sum_{n\geq 0}A_{\geq 0}(n)x^n=\sum_{n\geq 0}|\VHC(\Av_n(312,321))|x^n.\] We have \[\sum_{a\geq 0}\sum_{n\geq 0}A_a(n+1)x^nz^a=\sum_{a\geq 0}\sum_{n\geq 0}A_{\geq a}(n+1)x^nz^a-\sum_{a\geq 0}\sum_{n\geq 0}A_{\geq a+1}(n)x^nz^a\]
\begin{equation}\label{Eq14}
=\frac{J(x,z)-1/(1-z)}{x}-\frac{J(x,z)-J(x,0)}{z}
\end{equation}
and 
\[\sum_{a\geq 0}\sum_{n\geq 0}A_{a-1}(n+1)x^nz^a=\sum_{a\geq 1}\sum_{n\geq 0}A_{a-1}(n+1)x^nz^a=\sum_{a\geq 1}\sum_{n\geq 0}A_{\geq a-1}(n+1)x^nz^a-\sum_{a\geq 1}\sum_{n\geq 0}A_{\geq a}(n)x^nz^a\]
\begin{equation}\label{Eq15}
=\frac{z}{x}(J(x,z)-1/(1-z))-(J(x,z)-J(x,0)).
\end{equation}

Combining \eqref{Eq13} and \eqref{Eq15} gives \[\sum_{a\geq 0}\sum_{n\geq 0}A_a(n+1)x^nz^a=\sum_{a\geq 0}\sum_{n\geq 0}A_{a-1}(n+1)x^nz^a+\sum_{a\geq 0}\sum_{n\geq 0}\sum_{r=1}^aA_{\geq a-r+1}(n-1)x^nz^a\] \[=\frac{z}{x}(J(x,z)-1/(1-z))-(J(x,z)-J(x,0))+x\sum_{a\geq 0}\sum_{r=1}^a\sum_{n\geq 0}A_{\geq a-r+1}(n)x^nz^a\] \[=\frac{z}{x}(J(x,z)-1/(1-z))-(J(x,z)-J(x,0))+\frac{x}{1-z}(J(x,z)-J(x,0))\] \[=J(x,z)\left(\frac{z}{x}+\frac{x}{1-z}-1\right)-\frac{z}{x(1-z)}+J(x,0)\left(1-\frac{x}{1-z}\right).\] We now combine this with \eqref{Eq14} to obtain the equation \[\frac{J(x,z)-1/(1-z)}{x}-\frac{J(x,z)-J(x,0)}{z}
=J(x,z)\!\left(\frac{z}{x}+\frac{x}{1-z}-1\right)-\frac{z}{x(1-z)}+J(x,0)\!\left(1-\frac{x}{1-z}\right).\] Rearranging this equation yields 
\begin{equation}\label{Eq16}
\frac{1}{z(1-z)^2}(J(x,z)-J(x,0))\left(x(1-z)^2+x^2z-z(1-z)^2\right)=J(x,0)-\frac{1}{1-z}.
\end{equation} 

We now employ the kernel method. There is a unique power series $Z=Z(x)$ such that 
\begin{equation}\label{Eq17}
x(1-Z)^2+x^2Z-Z(1-Z)^2=0\quad\text{and}\quad\dfrac{1}{1-Z(x)}=1+x+x^2+O(x^3).
\end{equation} If we substitute $z=Z(x)$ in \eqref{Eq16}, we find that $J(x,0)=\dfrac{1}{1-Z(x)}$. Now, \[J(x,0)-1-\frac{xJ(x,0)}{1-x^2J(x,0)^2}=\frac{1}{1-Z}-1-\frac{x\frac{1}{1-Z}}{1-x^2\frac{1}{(1-Z)^2}}=\frac{Z}{1-Z}-\frac{x(1-Z)}{(1-Z)^2-x^2}=0,\] where the last equality follows from the first equation in \eqref{Eq17}. This means that \[J(x,0)=1+\frac{xJ(x,0)}{1-x^2J(x,0)^2}.\] The theorem now follows from the Lagrange Inversion formula.  
\end{proof}

\begin{remark}\label{Rem3}
One can show that $\sum_{k=0}^{\left\lfloor\frac{n-1}{2}\right\rfloor}\frac{1}{2k+1}{n-k-1\choose k}{n\choose 2k}$ is also the number of Dyck paths of length $2n$ in which every string of consecutive down steps has odd length (see \cite{OEIS}). It would be interesting to have a bijection between the set of such Dyck paths and the set $\VHC(\Av_n(312,321))$. 
\end{remark}

\section{$\VHC(\Av(231,1243))$}\label{Sec:1243}

In \cite{DefantCatalan}, the current author considered uniquely sorted permutations avoiding one length-$3$ pattern and one length-$4$ pattern. He found connections between these uniquely sorted permutations and intervals in lattices of Dyck paths. He also gave several additional enumerative conjectures concerning uniquely sorted permutations avoiding a length-$3$ pattern and a length-$4$ pattern. In a similar vein, it seems promising to enumerate valid hook configurations of permutations that avoid a length-$3$ pattern and a length-$4$ pattern. The purpose of this section is to initiate this investigation by enumerating valid hook configurations of permutations that avoid $231$ and $1243$. To do this, we rely on a lemma that makes use of our connection between valid hook configurations and Motzkin paths. In what follows, let \[A_\ell(n)=|\VHC(\{\pi\in\Av_{n+\ell}(132,231):\tl(\pi)=\ell\})|,\] \[A_{\geq \ell}(n)=|\VHC(\{\pi\in\Av_{n+\ell}(132,231):\tl(\pi)\geq\ell\})|,\] and \[J(x,z)=\sum_{\ell\geq 0}\sum_{n\geq 0}A_{\geq\ell}(n)x^nz^\ell.\]

\begin{lemma}\label{Lem1}
We have \[J(x,z)=\frac{(-1+2z)(1-\sqrt{1-2x-3x^2})-x}{(1-z)(x(-2+z)+(1-\sqrt{1-2x-3x^2})z)}\]
\end{lemma}

\begin{proof}
Let $M(x)=\sum_{n\geq 0}M_nx^n=\dfrac{1-x-\sqrt{1-2x-3x^2}}{2x^2}$ be the generating function of the sequence of Motzkin numbers, and consider the generating function $F(x,z)=\dfrac{xz^2(1+xM(x))}{1-xz(1+xM(x))}$. For $n\geq 1$, let $\mathfrak a(n,\ell)$ be the number of Motzkin paths of length $n$ in which $\ell$ endpoints of steps touch the horizontal axis. Also, let $\mathfrak b(n,\ell)$ be the number of Motzkin paths of length $n$ in which the first down step is the $\ell^\text{th}$ step (with the convention that $\mathfrak b(n,n+1)=1$). In \cite{Deutsch}, Deutsch described a simple involution on Dyck paths. By extending this involution in an obvious way to Motzkin paths, one can show that $\mathfrak a(n,\ell)=\mathfrak b(n,\ell)$. Furthermore, it is known \cite{OEIS} that $\mathfrak a(n,\ell)$ is the coefficient of $x^nz^\ell$ in $F(x,z)$. Hence, $\mathfrak b(n,\ell)$ is the coefficient of $x^nz^\ell$ in $F(x,z)$. 

We now refer back to the proof of Theorem \ref{Thm3}. If $\Lambda$ is a Motzkin path of length $n-1$ and $\ell\in\{0,\ldots,n-2\}$, then the first down step in $\Lambda$ is the $(\ell+1)^\text{st}$ step if and only if the tail length of the underlying permutation of $\widehat\DL_n^{-1}(\Lambda,\Lambda)$ is $\ell$. Moreover, the bijection $\VHC(\Av(132,231))\to\VHC(\Av(231,312))$ from \cite{DefantFertilityWilf} preserves lengths and tail lengths of the underlying permutations of valid hook configurations. It follows that $A_\ell(n-\ell)=\mathfrak b(n-1,\ell+1)$ whenever $n\geq 3$ and $\ell\in\{1,\ldots,n-2\}$. We also have $A_\ell(0)=1$ and $A_\ell(1)=0$ for all $\ell\geq 0$. After putting this all together, we find that 
\begin{equation}\label{Eq18}
\sum_{\ell\geq 0}\sum_{n\geq 0}A_\ell(n-\ell)x^nz^\ell=1+xz+x^2\left(\frac{1}{xz}F(x,z)-\frac{z(1-z)}{1-xz}\right).
\end{equation} Let $\widetilde F(x,z)$ be the generating function in \eqref{Eq18}. Straightforward manipulations allow us to find that 
\begin{equation}\label{Eq19}
J(x,z)=\frac{(z/x)\widetilde F(x,z/x)-J(x,0)}{z/x-1}.
\end{equation} It follows from Theorem \ref{Thm3} that $J(x,0)=1+xM(x)$. The remainder of the proof now amounts to substituting the relevant expressions into \eqref{Eq19} and simplifying.   
\end{proof}

\begin{theorem}\label{Thm7}
We have \[\sum_{n\geq 0}|\VHC(\Av_n(231,1243))|x^n=1+\frac{2x^2}{3x-1+\sqrt{1-2x-3x^2}}.\]
\end{theorem}

\begin{proof}
Let \[\mathcal D_\ell(n)=\{\pi\in\Av_{n+\ell}(231,1243):\tl(\pi)=\ell\}\quad\text{and}\quad \mathcal D_{\geq\ell}(n)=\{\pi\in\Av_{n+\ell}(231,1243):\tl(\pi)\geq\ell\}.\] Let $B_\ell(n)=|\VHC(\mathcal D_\ell(n))|$ and $B_{\geq\ell}(n)=|\VHC(\mathcal D_{\geq\ell}(n))|$. 

Suppose $\pi\in\mathcal D_\ell(n+1)$ is such that $\pi_{n+1-i}=n+1$ (where $n\geq 0$). Then $n+1-i$ is a tail-bound descent of $\pi$. 
We can use Corollary \ref{Cor1} to see that $|\VHC(\pi)|$ is equal to the number of triples $(H,\mathcal H_U,\mathcal H_S)$, where $H\in\SW_{n+1-i}(\pi)$, $\mathcal H_U\in \VHC(\pi_U^H)$, and $\mathcal H_S\in \VHC(\pi_S^H)$. Choosing $H$ amounts to choosing the number $j\in\{1,\ldots,\ell\}$ such that the northeast endpoint of $H$ is $(n+1+j,n+1+j)$. The permutation $\pi$ and the choice of $H$ determine the permutations $\pi_U^H$ and $\pi_S^H$. On the other hand, the choices of $H$ and the permutations $\pi_U^H$ and $\pi_S^H$ uniquely determine $\pi$. It follows that $B_\ell(n+1)$, which is the number of ways to choose an element of $\VHC(\mathcal D_\ell(n+1))$, is also the number of ways to choose $j$, the permutations $\pi_U^H$ and $\pi_S^H$, and the valid hook configurations $\mathcal H_U$ and $\mathcal H_S$. Let us fix a choice of $j$. 

First, suppose $i\leq n-1$. Because $\pi$ avoids $231$ and $1243$, one can show that \[\pi_U^H=(n-i)(n-i-1)\cdots 1(n+1)(n+2+j)(n+3+j)\cdots(n+\ell+1),\] while $\pi_S^H$ must be a permutation of $\{n-i+1,\ldots,n+j\}\setminus\{n+1\}$ that avoids $132$ and $231$. The valid hook configuration $\mathcal H_U$ has $n-i-1$ hooks; choosing this valid hook configuration amounts to choosing the northeast endpoints of these hooks from the top $\ell-j+1$ points in $\pi_U^H$. Thus, the number of choices for $\mathcal H_U$ is ${\ell-j+1\choose n-i-1}$. The normalization of $\pi_S^H$ is in $\Av_{i+j-1}(132,231)$ and has tail length at least $j-1$. Any permutation in $\Av_{i+j-1}(132,231)$ and has tail length at least $j-1$ can be chosen as the normalization of $\pi_S^H$. Also, every permutation has the same number of valid hook configurations as its normalization. Consequently, the number of choices for $\pi_S^H$ and $\mathcal H_S$ is $A_{\geq j-1}(i)$, where we have preserved the notation immediately preceding Lemma \ref{Lem1}. 

If $i=n$, then we repeat the same argument, except that there is only one choice for $\pi_U^H$ and $\mathcal H_U$ and that the number of choices for $\pi_S^H$ and $\mathcal H_S$ is $B_{\geq j-1}(n)$. We now obtain the recurrence  
\[B_\ell(n+1)=\sum_{j=1}^\ell\sum_{i=1}^{n-1}{\ell-j+1\choose n-i-1}A_{\geq j-1}(i)+\sum_{j=1}^\ell B_{\geq j-1}(n).\]
After multiplying this equation by $x^nz^\ell$, summing over $\ell\geq 0$ and $n\geq 0$, and simplifying, we find that 
\begin{equation}\label{Eq21}
\sum_{\ell\geq 0}\sum_{n\geq 0}B_\ell(n+1)x^nz^\ell=\frac{z}{1-z}\left(I(x,z)-\frac{1}{1-z}\right)+\frac{xz(1+x)}{1-(1+x)z}\left(J(x,z)-\frac{1}{1-z}\right),
\end{equation} where $I(x,z)=\sum_{\ell\geq 0}\sum_{n\geq 0}B_{\geq\ell}(n)x^nz^\ell$ and $J(x,z)$ is as in Lemma \ref{Lem1}. The same argument used to derive \eqref{Eq5} in the proof of Theorem \ref{Thm1} shows that 
\begin{equation}\label{Eq20}
\sum_{\ell\geq 0}\sum_{n\geq 0}B_\ell(n+1)x^nz^\ell=\frac{I(x,z)-1/(1-z)}{x}-\frac{I(x,z)-I(x,0)}{z}.
\end{equation} 
We now put \eqref{Eq21} and \eqref{Eq20} together and rearrange terms to find that 
\begin{equation}\label{Eq22}
I(x,z)\left(1-\frac{z}{x}+\frac{z^2}{1-z}\right)=I(x,0)-T(x,z),
\end{equation} where \[T(x,z)=\frac{z}{x(1-z)}-\frac{z^2}{(1-z)^2}+\frac{xz^2(1+x)(J(x,z)-1/(1-z))}{1-(1+x)z}.\]

As in previous proofs, we now use the kernel method. Let $Z(x)=\dfrac{x(1+xM(x))}{1+x}$, where $M(x)=\sum_{n\geq 0}M_nx^n=\dfrac{1-x-\sqrt{1-2x-3x^2}}{2x^2}$ is the Motzkin generating function. We have $\displaystyle 1-\frac{Z(x)}{x}+\frac{Z(x)^2}{1-Z(x)}=0$, so substituting $z=Z(x)$ into \eqref{Eq22} yields $I(x,0)=T(x,Z(x))$. The expression $T(x,Z(x))$ simplifies to $\displaystyle 1+\frac{2x^2}{3x-1+\sqrt{1-2x-3x^2}}$. Finally, \[I(x,0)=\sum_{n\geq 0}B_{\geq 0}(n)x^n=\sum_{n\geq 0}|\VHC(\Av_n(231,1243))|x^n. \qedhere\]
\end{proof}

\section{Concluding Remarks and Future Directions}
Most of the present article has concerned sets of the form $\VHC(\Av(\tau^{(1)},\ldots,\tau^{(r)}))$, where $\tau^{(1)},\ldots,$ $\tau^{(r)}\in S_3$. These sets are completely uninteresting when one of the permutations $\tau^{(1)},\ldots,\tau^{(r)}$ is $123$ or $213$, so we can restrict our attention to the cases in which $\tau^{(1)},\ldots,\tau^{(r)}\in \{132,231,312,321\}$. We have said nothing about this problem when $r\geq 3$, but that is because the enumerative results are not terribly difficult in these cases. For completeness, we state these results (without proof) in the following proposition. Let $F_n$ denote the $n^\text{th}$ Fibonacci number (with $F_1=F_2=1$). 

\begin{proposition}\label{Prop1}
We have \[\sum_{n\geq 0}|\VHC(\Av_n(231,312,321))|x^n=\frac{1 - x + x^2 - \sqrt{1 - 2x - x^2 - 2x^3 + x^4}}{2x^2}.\]
For every $n\geq 1$, we have 
\[|\VHC(\Av_n(132,231,321))|=|\VHC(\Av_n(132,312,321))|=1+{n-1\choose 2},\] 
\[|\VHC(\Av_n(132,231,312))|=F_n,\] and \[|\VHC(\Av_n(132,231,312,321))|=n-1.\] 
\end{proposition}

We have said nothing about the numbers $|\VHC(\Av_n(321))|$; it would be interesting to have nontrivial results concerning these numbers or their generating function. We also wish to remind the reader of the combinatorial proof requested in Remark \ref{Rem3}. 

Finally, we believe it could be interesting to enumerate valid hook configurations of permutations avoiding collections of length-$4$ patterns. One could also enumerate valid hook configurations of permutations avoiding one length-$3$ pattern and one length-$4$ pattern. We initiated this direction in Section \ref{Sec:1243} when we enumerated $\VHC(\Av(231,1243))$. We also have the following conjecture. 

\begin{conjecture}\label{Conj2}
We have \[\sum_{n\geq 0}|\VHC(\Av_n(132,3241))|x^n=\sum_{n\geq 0}|\VHC(\Av_n(231,2143))|x^n=\frac{1+x^2-\sqrt{1 - 4x + 2x^2 + x^4}}{2x}.\]
\end{conjecture}

Let us remark that the first equality in Conjecture \ref{Conj2} follows from the results in \cite{DefantFertilityWilf}. Thus, the actual content of the conjecture lies in the explicit form of the generating function.

\section{Acknowledgments}
The author was supported by a Fannie and John Hertz Foundation Fellowship and an NSF Graduate Research Fellowship.

\end{document}